\documentclass[letterpaper, 11pt,  reqno]{amsart}

\usepackage[margin=1.1in,marginparwidth=1.5cm, marginparsep=0.5cm]{geometry}

\usepackage{amsmath,amssymb,amscd,amsthm,amsxtra, esint}

\usepackage{mathrsfs} 

\usepackage[implicit=true]{hyperref}

\usepackage{color} 
\usepackage{ulem}
\usepackage[makeroom]{cancel}

\allowdisplaybreaks[2]

\definecolor{gr}{rgb}   {0.,   0.69,   0.23 }
\definecolor{bl}{rgb}   {0.,   0.5,   1. }
\definecolor{mg}{rgb}   {0.85,  0.,    0.85}
\definecolor{yl}{rgb}   {0.8,  0.7,   0.}
\definecolor{or}{rgb}  {0.7,0.2,0.2}

\newtheorem{theorem}{Theorem} [section]

\newtheorem{lemma}[theorem]{Lemma}

\newtheorem{remark}[theorem]{Remark}

\newcommand{\noi}{\noindent}
\newcommand{\Z}{\mathbb{Z}}
\newcommand{\R}{\mathbb{R}}

\newcommand{\T}{\mathbb{T}}

\newcommand{\F}{\mathcal{F}}

\newcommand{\al}{\alpha}
\newcommand{\be}{\beta}
\newcommand{\dl}{\delta}

\newcommand{\eps}{\varepsilon}
\newcommand{\kk}{\kappa}
\newcommand{\g}{\gamma}
\newcommand{\G}{\Gamma}
\newcommand{\ld}{\lambda}

\newcommand{\s}{\sigma}

\newcommand{\ft}{\widehat}

\newcommand{\cj}{\overline}
\newcommand{\dx}{\partial_x}
\newcommand{\dt}{\partial_t}
\newcommand{\dd}{\partial}

\newcommand{\too}{\longrightarrow}

\renewcommand{\l}{\ell}

\newcommand{\les}{\lesssim}

\newcommand{\jb}[1]
{\langle #1 \rangle}

\newcommand{\ind}{\mathbf 1}

\newcommand{\M}{\mathcal{M}}

\newcommand{\Hi}{\mathcal{H}}

\newtheorem*{ackno}{Acknowledgements}

\numberwithin{equation}{section}
\numberwithin{theorem}{section}

\DeclareMathOperator{\sgn}{sgn}

\newcommand{\Qdl}{\mathcal{Q}_{\dl}}

\newcommand{\vk}{\tau}

\makeatletter
\@namedef{subjclassname@2020}{%
  \textup{2020} Mathematics Subject Classification}
\makeatother

\newcommand{\Gdl}{\mathcal{G}_{\dl} }
\newcommand{\Tdl}{\mathcal{T}_{\dl} }

\begin{document}
\baselineskip = 14pt


\title[ILW in negative Sobolev spaces]
{Intermediate long wave equation\\in negative Sobolev spaces}

\author[A.~Chapouto, J.~Forlano,  G.~Li, T.~Oh,  and D.~Pilod]
{Andreia Chapouto, Justin Forlano, Guopeng Li, \\Tadahiro Oh, and Didier Pilod}

\address{
Andreia Chapouto\\ School of Mathematics\\
The University of Edinburgh\\
and The Maxwell Institute for the Mathematical Sciences\\
James Clerk Maxwell Building\\
The King's Buildings\\
Peter Guthrie Tait Road\\
Edinburgh\\ 
EH9 3FD\\
United Kingdom}

\email{a.chapouto@ed.ac.uk}

\address{
Justin Forlano\\ School of Mathematics\\
The University of Edinburgh\\
and The Maxwell Institute for the Mathematical Sciences\\
James Clerk Maxwell Building\\
The King's Buildings\\
Peter Guthrie Tait Road\\
Edinburgh\\ 
EH9 3FD\\
United Kingdom}

\email{j.forlano@ed.ac.uk}

\address{
Guopeng Li\\ School of Mathematics\\
The University of Edinburgh\\
and The Maxwell Institute for the Mathematical Sciences\\
James Clerk Maxwell Building\\
The King's Buildings\\
Peter Guthrie Tait Road\\
Edinburgh\\ 
EH9 3FD\\
United Kingdom}

\email{guopeng.li@ed.ac.uk}

\address{
Tadahiro Oh\\ School of Mathematics\\
The University of Edinburgh\\
and The Maxwell Institute for the Mathematical Sciences\\
James Clerk Maxwell Building\\
The King's Buildings\\
Peter Guthrie Tait Road\\
Edinburgh\\ 
EH9 3FD\\
United Kingdom}

\email{hiro.oh@ed.ac.uk}

\address{
Didier Pilod
Department of Mathematics\\
 University of Bergen\\ Postbox 7800\\
  5020 Bergen\\
   Norway}

\email{Didier.Pilod@uib.no}

\subjclass[2020]{35Q35, 37K10, 76B55}

\keywords{intermediate long wave equation; Benjamin-Ono equation; a priori bound; complete integrability;
ill-posedness}

\begin{abstract}
We study the intermediate long wave equation (ILW) in negative Sobolev spaces.
In particular, despite the lack of scaling invariance, 
we identify the regularity $s = -\frac 12$
as the critical regularity for ILW with any depth parameter, 
by establishing the following two results.
(i)~By viewing ILW as a perturbation of the Benjamin-Ono equation (BO)
and exploiting the complete integrability of BO, 
we  establish a global-in-time a priori bound
on the $H^s$-norm of a solution to  ILW for   $ - \frac 12 < s < 0$.
(ii)~By making use of explicit solutions, we  prove that ILW is ill-posed in $H^s$ for  $s < - \frac 12$.
Our results apply to both  the real line case and the periodic case.

\end{abstract}


%
\maketitle

\vspace{-7mm}

\tableofcontents

\vspace{-7mm}

\section{Introduction}

We consider the intermediate long wave equation (ILW)
on $\M = \R$ or $\T = (\R/ \Z)$:
\begin{align}
\begin{cases}
\dt u - \mathcal{G}_{\dl}\dx^2 u =\dx(u ^2) \\
u \vert_{t=0}  =u_0,
\end{cases}
\quad 
(t,x)\in \R\times \mathcal{M} \label{ILW}
\end{align}

\noi
for $0<\dl<\infty$.
The operator $\Gdl$ is given by 
\begin{align}
\mathcal{G}_{\dl}=\Tdl - \dl^{-1} \dx^{-1}, 
\label{G1}
\end{align}

\noi
where $\Tdl$ is the  Fourier multiplier operator with symbol
\begin{align}
\ft{ \Tdl f}(\xi) = -i \coth(\dl \xi) \ft f(\xi), \quad \xi \in \ft{\M}. \label{Tdl}
\end{align}

\noi
Here, $\ft \M$ denotes 
the Pontryagin dual of $\M$, i.e.
$\ft \M = \R$ if $\M = \R$, 
and 
$\ft \M = \Z$ if $\M = \T$.
The ILW equation \eqref{ILW} was
introduced in \cite{joseph, KKD} 
as a model describing the propagation of an internal wave at the interface of a stratified fluid of 
finite depth $\dl$, with further applications in modeling wave phenomena in oceanography and meteorology. 
Furthermore, 
it appears as an ``intermediate'' equation of finite depth $0 < \dl < \infty$ 
between the Benjamin-Ono equation (deep-water limit: $\dl\to\infty$)
and the  KdV equation (shallow-water limit: $\dl\to0$), 
attracting wide attention 
from both the applied and theoretical scientific communities.
Even from the purely analytical point of view, 
 \eqref{ILW} is of great  interest due to its rich structure;
 it is a dispersive equation, admitting soliton solutions.
 Moreover,  it is completely integrable with an infinite number of conservation laws.
See \cite{S19, KS} for an overview of these topics and more on the physical significance of ILW.

Despite  recent popularity of the ILW equation and its deep connection to the well-known 
Benjamin-Ono and KdV  equations, there remain many open questions in  well-posedness of~\eqref{ILW} and 
convergence as $\dl \to 0$ or $\infty$.
In this paper, we focus on the former question;
see  \cite{ABFS, MV, MPV, IS23, CLOP} for the known well-posedness results
for ILW.
See also \cite{ABFS, Gli, LOZ, CLOP, CLOZ}
for results on convergence issues from both deterministic and statistical viewpoints.

It is known (see \cite{MST, KTz2}) that,  
just like the Benjamin-Ono equation (see \eqref{BO} below), 
ILW~\eqref{ILW}
is  quasilinear in the sense that 
a contraction argument 
can not be used 
for constructing solutions,
which makes the  well-posedness question 
rather challenging, especially in a low-regularity setting.\footnote{We point out that, 
under the mean-zero assumption, the periodic BO and ILW
posed on the circle~$\T$ may be semilinear just like the KdV equation, 
at least in a smooth setting.
See \cite[Theorem 1.2]{Moli2}, 
where, for $s\ge 0$,  the solution map for BO \eqref{BO} on~$\T$ was shown to be real-analytic
from the subspace $H^s_0(\T) \subset H^s(\T)$, consisting  of mean-zero functions, 
into itself.
At this point, however, there is no known well-posedness
argument via a contraction argument 
for the mean-zero periodic BO and ILW.}
In ~\cite{IS23}, Ifrim and Saut
proved global well-posedness of ILW \eqref{ILW}
in $L^2(\R)$.
In a recent preprint \cite{CLOP}, the first, third, fourth, and fifth 
authors
provided a unified argument for $L^2$-global
well-posedness of~\eqref{ILW} on both the real line and the circle.
We point out that the basic strategy
in  \cite{IS23, CLOP} is to  view  ILW~\eqref{ILW} as
a perturbation of the Benjamin-Ono equation (BO):
\begin{align}
\dt u - \Hi \dx^2 u = \dx (u^2),
 \label{BO}
\end{align}

\noi
where $\Hi$ denotes the usual Hilbert transform with multiplier\footnote{On $\T$, we set $\sgn(0)=0$.}
$ -i \sgn(\xi)$, $\xi\in\ft{\M}$, 
and to suitably adapt the known well-posedness arguments
for the BO equation \cite{IK, MP, IT1}.
We will elaborate this viewpoint further in the following.

Our main goal in this paper is to study issues related to  well-posedness
of ILW \eqref{ILW}
in negative Sobolev spaces.
It is well known that a scaling symmetry, if it exists, 
provides an important threshold (called a scaling critical regularity)
on well-posedness for a dispersive equation.
For example, 
BO \eqref{BO} on the real line 
is known to be invariant under  the following $\dot H^{-\frac12}$-invariant scaling:
\begin{align}
u_{\ld}(t,x) = \ld^{-1}u(\ld^{-2}t, \ld^{-1}x), \quad \ld > 0.
\label{ILWscale}
\end{align}

\noi
This scaling symmetry induces 
 the scaling critical regularity $s= - \frac 12$ for BO.
While  ILW does not enjoy a scaling symmetry, 
it was remarked 
 in \cite[Remark 4.2]{CLOP}
 that if $u$ is a solution to ILW~\eqref{ILW} on $\R$
 (with the depth parameter $\dl$), 
 then  the rescaled function $u_\ld$ in \eqref{ILWscale}
solves~\eqref{ILW} with the  depth parameter $\ld \dl$. 
Namely, {\it the family of the ILW equations
with depth parameters $0 < \dl < \infty$
is invariant under the scaling \eqref{ILWscale}.}
This observation hints that the regularity $s= -\frac 12$
may be critical for ILW in an appropriate sense.
We show that this is indeed the case
by establishing the following results.
%

\begin{theorem}\label{THM:1}
Let $\M = \R$ or $\T$ and $0<\dl<\infty$.  Then, the following statements hold.
	
\smallskip
	
\noi
{\rm(i)} \textup{(global-in-time a priori bound).} Let $ - \frac 12< s < 0$.  Given $u_0 \in
H^\infty(\M)$, let $u$ be the {\rm(}unique{\rm)} smooth solution to the ILW equation
\eqref{ILW}.
Then, given small $\eps > 0$, 
there exist positive constants
$C_s$ and 
 $A_{\dl, s} \sim
 \dl^{-2}(1+\dl^{-|s|- \frac 12 - \eps})$, 
 independent of $u_0 \in H^\infty(\M)$, 
such that 
\begin{align}
\|u(t)\|_{H^{s}} 
\le C_s^{|s|+ 1}
e^{A_{\dl, s}|t|}
\Big(1+2C_s e^{A_{\dl, s}|t|} \|u_0\|_{H^{s}}\Big)^{\frac{2|s|}{1-2|s|}}
\|u_0\|_{H^{s}}
 \label{aprioribd}
\end{align}

\noi
for any $t \in \R$.
	
	\smallskip
	
	\noi
{\rm(ii)} \textup{(ill-posedness)} Let $s < - \frac 12$. Then, the ILW
equation \eqref{ILW} is ill-posed in $H^s(\M)$. 
Moreover, when $\M = \T$, 
given any $\al \in \R$, 
the ILW
equation \eqref{ILW} is ill-posed in $H^s_\al(\T)$, 
where 
 $H^s_\al(\T)$
denotes the subspace of $H^s(\T)$ consisting of functions
with spatial mean $\al$.

	
\end{theorem}


It follows from the proof of Theorem~\ref{THM:1}\,(ii)
that 
if the solution map $\Phi:  H^{s}(\M) \to  C([-T,T];H^{s}(\M))$,
sending initial data $u_0$ to solutions $u = \Phi(u_0)$ of ILW \eqref{ILW}, 
extended to $s < -\frac 12$, 
then it would be discontinuous
at $u_0 = - 2\pi \dl_0$ for any $T > 0$, 
where $\dl_0$ denotes the Dirac delta function.
The known global well-posedness of ILW in $L^2(\M)$ \cite{IS23, CLOP}
and the ill-posedness result in $H^s(\M)$ for $s < -\frac 12$ (Theorem \ref{THM:1}\,(ii))
leave the gap $-\frac{1}{2}\le s<0$.
While the  a priori bound in Theorem~\ref{THM:1}\,(i)
indicates that  well-posedness should extend, 
at least, to  the range $-\frac{1}{2}< s<0$, 
the actual well-posedness of ILW in the  range $-\frac{1}{2}\le s<0$ on either geometry is 
completely open.
In view of the positive and negative results in Theorem \ref{THM:1}, 
we propose that $s = -\frac 12$ is the critical regularity 
for the ILW equation \eqref{ILW}, 
which is in particular independent of the depth parameter $\dl$.

Let us briefly discuss the strategy for proving  Theorem \ref{THM:1}.
As for  the ill-posedness claim in Theorem~\ref{THM:1}\,(ii), 
we follow closely 
the strategy in \cite{BL, AH} 
for ill-posedness of BO in $H^{s}(\M)$,  $s<-\frac{1}{2}$; see also \cite{KPV3}. 
Namely, we make use of explicit traveling wave solutions to ILW~\eqref{ILW} 
which approximate, at time $t =0$,  (a constant multiple of) the Dirac delta function as the speed of the wave diverges to infinity. 
On the circle, the BO equation \eqref{BO}
is known to be ill-posed in the critical space $H^{-\frac 12}(\T)$
whose proof 
 heavily relies on the complete integrability
via the use of
the Birkhoff map; see \cite[Section 7]{GKT}. 
It would be of interest to investigate
if a similar ill-posedness result in $H^{-\frac 12}(\T)$ holds
for ILW.

Let us now turn to Theorem \ref{THM:1}\,(i).
While ILW is known to be completely integrable, 
we do not make use of  its integrable structure (which is not well understood)
 to prove Theorem~\ref{THM:1}(i). 
 We instead view ILW \eqref{ILW} as a perturbation of BO \eqref{BO} (just as in \cite{IS23, CLOP})
 and exploit the integrable structure of the BO equation.
Define   $\Qdl$  by 
\begin{align}
\Qdl = (\Tdl-\Hi)\dx, 
\label{Qdl}
\end{align}

\noi
where $\Tdl$ is as in \eqref{Tdl}. 
Then, in view of \eqref{G1}, 
we can write ILW  \eqref{ILW}
as
\begin{align}
\dt u - \Hi \dx^2 u =   \dx (u^2)
- \dl^{-1}\dx u + \Qdl\dx u.
\label{ILWQ}
\end{align}


\noi
As seen in \cite{IS23, CLOP},\footnote{In \cite{IS23, CLOP}, a Galilean transform (see \eqref{G2} below) was applied to remove 
the linear term $- \dl^{-1}\dx u$ in \eqref{ILWQ}.
While we could apply the same Galilean transform to remove this term
and study a renormalized ILW, 
it turns out that such a procedure 
 is not necessary for establishing an a priori bound, 
since the generator for this linear term (i.e.~$\dl^{-1} M(u)$, where $M(u)$ is the mass
defined in~\eqref{be1}) Poisson-commutes with the key quantity $\be_s(\kk; u)$ 
defined in \eqref{be2};
see Lemma \ref{PROP:BOLAX}.}  
the operator  $\Qdl$ 
enjoys a strong smoothing property (see Lemma~\ref{LEM:Qdl}), 
which allows us to view the last term in \eqref{ILWQ}
as a perturbation in a suitable sense.

There have been two successful approaches to the well-posedness
study of BO, exploiting its complete integrability.
In \cite{GKT}, G\'erard, Kappeler, and Topalov
proved sharp global well-posedness of the periodic BO in $H^s(\T)$, $s > -\frac 12$, 
by building a suitable Birkhoff map.
In a recent preprint~\cite{KLV}, 
Killip, Laurens, and Vi\c{s}an
applied the method of commuting flows \cite{KV}
to  the BO equation \eqref{BO}
and 
proved 
sharp\footnote{Except for the endpoint $s = -\frac 12$ on the real line.}
global well-posedness 
in $H^s(\M)$, $s > -\frac 12$, 
on both the real line and the circle.
In the following, we use
 the completely integrable structure for BO as presented in the latter work \cite{KLV}.
In \cite{KLV}, 
for  $\kk\gg 1$,
the authors constructed 
the quantity  $\be_{s}(\kk; u)$ (see~\eqref{be2} below)
which 
 is conserved under the flow of BO
  and is equivalent to the $H^{s}$-norm, provided that  $-\frac12<s<0$;
  see Lemma~\ref{PROP:BOLAX} below.
While this quantity $\be_{s}(\kk; u)$ is 
not conserved under the flow of 
ILW \eqref{ILWQ},   
 the only non-zero contribution 
 to its time derivative 
 comes 
from  the last term 
$\Qdl\dx u$
in \eqref{ILWQ}. 
As this term is linear and enjoys sufficient smoothing,
we can apply  a Gronwall argument to control  the growth of $\be_{s}(\kk;u)$. 
This explains 
 the reason for the time-dependent growth in \eqref{aprioribd}.

We conclude this introduction with several  remarks.

\begin{remark}\rm
(i) 
While we expect that there are time-independent a priori bounds for ILW, 
it seems that one would have to develop an appropriate completely integrable structure of ILW
for this purpose.
We chose not to pursue this direction to exemplify the point that ILW can be thought of 
as a perturbation of BO to obtain the a priori bound.
We  note that (the proof of) 
the a priori bound in Theorem \ref{THM:1}
also  holds for suitable (potentially) non-integrable variants of BO; see, for example, 
Part~(ii) of this remark and also
Remark~\ref{REM:2fluid}.

\smallskip

\noi
(ii) A close look at  the proof of 
Theorem \ref{THM:1}\,(i)
(see \eqref{J2a} in the proof of 
Lemma~\ref{LEM:bds1}) shows that
we only need smoothing\footnote{Namely, 
mapping $L^2(\M)$ into $H^{\frac{3}{2}-s+\eps}(\M)$.} of 
order $\frac{3}{2}-s+\eps$ (for some $\eps>0$)
from the operator $\Qdl $
in proving  the a priori bound \eqref{aprioribd}.
This in particular implies that if we instead consider 
the following variant of the BO equation:
\begin{align}
\dt u - \Hi \dx^2 u =   \dx (u^2)
+ c_1 \dx u + c_2 I \dx u, 
\label{ILWQ2}
\end{align}

\noi
where $c_1,  c_2 \in \R$
and $I$ is a linear operator with smoothing of order  $\frac{3}{2}-s+\eps$ for some $\eps > 0$, 
then a slight modification of the proof of Theorem \ref{THM:1}\,(i)
yields an analogous 
a priori bound on the $H^s$-norm
of a solution to \eqref{ILWQ2} for $- \frac 12 < s < 0$.
We point out that, under a weaker assumption
on $I$ being smoothing of order $1$, 
a slight modification of the argument in 
\cite{CLOP} 
yields global well-posedness of 
\eqref{ILWQ2} in $L^2(\M)$.
Our proof of ill-posedness in Theorem \ref{THM:1}\,(ii)
relies on the explicit solutions to ILW \eqref{ILW}
and thus it does not extend to \eqref{ILWQ2}.

\smallskip

\noi
(iii)
By  using a differencing technique as in \cite{KVZ}, 
we expect that our approach of building a (time-dependent) a priori bound (as in Theorem~\ref{THM:1}\,(i))
will extend to  positive regularities.
See also \cite{MV, CLOP}
for persistency-of-regularity
arguments, controlling the $H^s$-norms of solutions
to~\eqref{ILW}, at least for $0 < s < 1$.

\smallskip

\noi
(iv)
Following \cite{KVZ}, 
a  quantity based on a series expansion of the perturbation determinant
was used in~\cite{Talbut} to establish an a priori bound
on the $H^s$-norm of a solution to  BO for $- \frac 12< s < 0 $.
See also  \cite{OW1, KlausSchippa}.
For our purpose, however, we find the quantity $\be_s(\kk;u)$ in \eqref{be2} more convenient
especially because it does not involve a series expansion.

\smallskip

\noi
(v)
We point out  a similarity 
between 
our argument for establishing the a priori bound (Theorem~\ref{THM:1}\,(i))
and the work of Laurens~\cite{TL1}
 who 
 studied low-regularity well-posedness of the KdV equation with a space-time potential.
The presence of the potential also broke the conservation laws and thus a Gronwall argument was needed to control their growth.

\end{remark}

\begin{remark}\rm \label{REM:2fluid}
In \cite{KKD}, 
the equation for the motion of an internal wave in a finite depth fluid was derived with two depth parameters 
$\dl_j$, $j = 1, 2$, where $\dl_1$ and $\dl_2$ represent the depths of the upper and lower fluids, respectively, 
and is given by 
	\begin{align}
		\dt u - c_1\mathcal{G}_{\dl_1}\dx^2 u-c_2\mathcal{G}_{\dl_2}\dx^2 u =\dx(u ^2), \label{ILW2}
	\end{align}
where $c_1,c_2>0$. 
By applying the  Galilean transform\footnote{We point out that the Galilean transform \eqref{G2}
is needed only for  proving $L^2$-global well-posedness of \eqref{ILW2}, following the argument in \cite{CLOP}, 
and that 
it is not needed to establish an a priori bound in $H^s(\M)$, $-\frac 12 < s < 0$.}
\begin{align}
v(t,x) = u( t, x+ \g t), \quad \g : = \tfrac{c_1}{\dl_1}+\tfrac{c_{2}}{\dl_2}, 
\label{G2}
\end{align}

\noi
we see that $v$ satisfies the renormalized equation:
\begin{align}
\dt v - c_1\mathcal{T}_{\dl_1}\dx^2 v-c_2\mathcal{T}_{\dl_2}\dx^2 v =\dx(v ^2). \label{ILW22}
\end{align}

\noi
Then, we rewrite \eqref{ILW22} as
\begin{align}
\dt v - (c_1+c_2) \Hi \dx^2 v 
= \dx(v ^2)
+ c_1\mathcal{Q}_{\dl_1}\dx^2 v+c_2\mathcal{Q}_{\dl_2}\dx^2 v. 
\label{ILW23}
\end{align}

\noi
By viewing \eqref{ILW23} as 
a perturbation of the following BO equation:
	\begin{align}
\dt v - (c_1+c_2) \Hi \dx^2 v =\dx(v ^2),   
\label{BO2}
\end{align}

\noi
a slight modification of the argument in \cite{CLOP}
yields global well-posedness of \eqref{ILW22} (and of \eqref{ILW2})
in $L^2 (\M)$, and, moreover, the solutions converge to solutions of \eqref{BO2} as 
$\min (\dl_1, \dl_2) \to \infty$.
Similarly, 
a slight modification of the proof of Theorem \ref{THM:1}\,(i)
yields an a priori bound
on the $H^s$-norm of a solution to~\eqref{ILW2}
for $ - \frac 12 < s < 0$.
We point out that~\eqref{ILW2} is not expected to be completely integrable.

\end{remark}

\section{Notations}\label{SEC:preliminary}
We write $A\les B$ to denote that there exists $C>0$ 
such that $A \leq CB$, and $A\ll B$ when $A\leq C B$ with 
sufficiently small $C > 0$.

Next, we go over our convention for Fourier transforms, following~\cite{KLV}.
On the real line, we write 
\begin{align}
\ft f(\xi) = \frac{1}{\sqrt{2\pi}}\int_{\R}  f(x)e^{-i\xi x}dx 
\quad \text{and} \quad f(x) =\frac{1}{\sqrt{2\pi}}\int_{\R}  \ft f(\xi)e^{i\xi x}d\xi,
\label{FT1}
\end{align}

\noi
while on the circle, we set 
\begin{align*}
	\ft f(\xi)= \int_{\T}  f(x)e^{-i\xi x}dx \quad \text{and} \quad f(x)=\sum_{\xi \in 2\pi \Z}\ft f(\xi)e^{i\xi x}.
\end{align*}

\noi
Then,  Plancherel's identity takes the form
\begin{align*}
\|f\|_{L^2(\R)} = \| \ft f\|_{L^2(\R)} \quad \text{and} \quad 
\|f\|_{L^2(\T)}= \| \ft f\|_{L^2(2\pi\Z)} = \bigg( \sum_{\xi \in 2\pi \Z} |\ft f(\xi)|^{2}\bigg)^\frac{1}{2}, 
\end{align*}

\noi
and  $\ft{\dx f}(\xi)=i\xi \ft f(\xi)$. 
Given $s\in \R$ and $\kk>0$, we define the $L^2$-based Sobolev spaces 
$H_{\kk}^{s}(\R)$ and $H_{\kk}^{s}(\T)$ via
\begin{align*}
	\|f\|_{H^{s}_\kk(\R)} =\bigg(\int  \jb{\xi}_{\kk}^{2s} |\ft f(\xi)|^{2} d\xi\bigg)^\frac 12 
\quad \text{and} \quad  \|f\|_{H_{\kk}^{s}(\T)}
=\bigg( \sum_{\xi \in2 \pi \Z} \jb{\xi}_{\kk}^{2s} |\ft f(\xi)|^{2}\bigg)^\frac 12,
\end{align*}

\noi
where $\jb{\xi}_{\kk}=(\kk^2 + |\xi|^2)^\frac 12$;
see also \cite{KT, OW2}.
When $\kk = 1$, 
$H^s_\kk(\M)$ reduces to the standard $L^2$-based Sobolev space $H^s(\M)$.
For $s < 0$, 
the $H^{s}_\kk$-norm of $f$ is decreasing in $\kk$, 
which plays an important role in proving \eqref{besequiv} below;
see \cite[the proof of Lemma 4.3]{KLV}.
Moreover, for $s<0$ and $\kk \ge 1$, we have 
\begin{align}
\| f\|_{H^{s}} \leq \kk^{-s}\|f\|_{H^{s}_{\kk}}. 
\label{HsHsk}
\end{align}

We define the Cauchy-Szeg\H{o} projector $\Pi_+$ by setting
\begin{align*}
\ft{ \Pi_{+}f}(\xi) = \ind_{[0,\infty)}( \xi)\cdot \ft f(\xi).
\end{align*}

\noi
Then, the Hardy space 
 $H^{s}_{+}(\M)$ is defined by 
  $H^{s}_{+}(\M) = \Pi_+ H^{s}(\M)$. 
 Recall that 
 when $\M = \R$,  functions 
in   $H^{s}_{+}(\R)$ 
 are the boundary values (on the real line) of 
holomorphic functions on the upper half-plane,
and that 
 when $\M = \T$,  functions 
in   $H^{s}_{+}(\T)$ 
 are the boundary values (on the circle) of 
holomorphic functions on the unit disc.

 \medskip

Next, we record  the following smoothing property of the operator $\Qdl \dx$.

\begin{lemma}\label{LEM:Qdl}
Let $\M = \R$ or $\T$ and $0 < \dl < \infty$.  
Then, given  $s_1,s_2\in \R$ with $s_1 \le s_2$, 
there exists $C_{s_1 - s_2} > 0$, independent of  $0<\dl<\infty$, such that 
\begin{align}
\| \Qdl \dx f\|_{H^{s_2}(\M)} \le C_{s_1 - s_2} \,\dl^{-2}(1+\dl^{s_1 -s_2}) \|f\|_{H^{s_1}(\M)}.
 \label{Qdl0}
\end{align}

\end{lemma}

\begin{proof}
We proceed as in 
 the
proof of  \cite[Lemma~2.3]{CLOP}; see also   \cite[Lemma 2.2]{IS23}.
With a slight abuse of notation, 
let $\ft \Qdl(\xi)$ denote the multiplier 
for the operator $\Qdl$ in~\eqref{Qdl}.
Then,  we have 
\begin{align*}
\ft{\Qdl}(\xi) = \xi\big(\coth(\dl \xi)-\sgn(\xi)\big) = \frac{2|\xi|}{e^{2|\dl \xi|}-1}
\end{align*} 

\noi
for $\xi \ne 0$.
Then, \eqref{Qdl0}
follows from 
 noting that $x^{\s} \le C_\s( e^{2x} - 1)$ for any $x \ge  0$, 
 provided that $\s \ge 1$.
\end{proof}

\section{Global-in-time a priori bound}\label{SEC:apriori}

In this section, we present the proof of Theorem \ref{THM:1}\,(i)
by viewing  ILW~\eqref{ILWQ}
as a perturbation of the BO equation \eqref{BO}
and exploiting the completely integrable structure of BO.

\subsection{Completely integrable structure of the BO equation}
In this subsection, we recall  from \cite{KLV} the completely integrable structure of 
the BO equation~\eqref{BO} and relevant results. Note that our convention for the signs in \eqref{BO} differ from that in \cite{KLV}. In order to translate between the two, one should use the map $u\mapsto -u$.  
First, recall that BO is a Hamiltonian PDE with the Hamiltonian:
\begin{align}
\text{Hamiltonian:} & \ \ H_{\text{BO}}(u)  = \frac{1}{2} \int_{\M} u \mathcal{H} \dx u dx + \frac{1}{3}\int_{\mathcal{M}}u^3 dx, 
\label{Hamil}
\end{align}

\noi
where the Poisson bracket is given by 
\begin{align}
\{ F,G\}  = \int_\M \frac{\dd F}{\dd u} \dx \frac{\dd G}{\dd  u} dx.
\label{P1}
\end{align}

%
%

\noi
Moreover, BO is completely integrable
with a Lax pair $(L, B) = (L_u, B_u)$ given by 
\begin{align}
L =-i \dd_{x} +\Pi_{+}u
\quad \text{and}\quad 
B =  -i\dd_{x}^{2}+2\dd_{x} \Pi_{+}u - 2(\dx \Pi_{+}u)
\label{Lax1}
\end{align}

\noi
such that $\dt L = [B, L]$, 
when $u$ is a solution to BO.
We also denote by  
$L_0=-i\dd_x$  the Lax operator with the zero potential.
The following lemma summarizes
the basic properties of the Lax operator $L$
(see Part (i) below for the definition of $L= L_u$ with a proper domain)
and its resolvent
from \cite{KLV}; see Propositions~3.2, 4.1, 4.3, and~4.7 
and Lemma~4.11 in \cite{KLV}.

\begin{lemma}\label{PROP:BOLAX}
Let $\M = \R$ or $\T$,  $ -\tfrac 12 < s <  0$,  and  $ \s=\tfrac{1}{2}( \tfrac{1}{2}+s)\in (0, \tfrac 14)$. 
Then, there exists a constant $C_{s}\ge 1$ such that whenever
$u\in H^{s}(\M)$ satisfies
\begin{align}
\kk \geq C_{s} \big(1+ \|u\|_{H^{s}_{\kk}}\big)^{\frac{1}{2\s}} \label{kkcond}
\end{align}

\noi
for some $\kk \ge 1$, 
the following statements hold true.

\smallskip	
	
\begin{itemize}
\item[\textup{(i)}]  
There exists a unique self-adjoint, semi-bounded operator 
$L = L_u$ associated to the quadratic form
\begin{align*}
f\mapsto \jb{f, L_0 f}_{L^2}+\int_{\M} u(x) |f(x)|^2 dx
\end{align*}

\noi
with  $H^{\frac 12}_{+}(\M)$
as the domain for the quadratic form,
where $\jb{\cdot, \cdot}_{L^2}$ denotes 
 the $L^2$-inner product  given by 
$\jb{f,g}_{L^2} = \int_\M \cj{f(x)} g(x)dx$.
The resolvent $R(\kk;u)=(L+\kk)^{-1}$ exists and maps $H^{-\frac 12}_{+}(\M)$ into $H^{\frac 12}_{+}(\M)$.

\smallskip
\item[\textup{(ii)}] 
%
Let 
 $m(\kk;u)=-R(\kk;u)\Pi_{+}u$.
 Then, we have 
\begin{align}
\|m(\kk; u)\|_{H^{s+1}_{\kk}} \les \|u\|_{H^{s}_{\kk}}
\quad\text{and}\quad  \|m(\kk; u)\|_{H^{s}}\les \kk^{-1}\|u\|_{H^{s}}. \label{mbds}
\end{align}
Moreover, if $u\in H^{\infty}(\M)$, then $m\in H^{\infty}(\M)$.

\smallskip

\item[\textup{(iii)}] The quantity $\be(\kk;u)$
defined by 
\begin{align*}
\be(\kk;u) = -\int u(x) m(x;\kk,u)dx=  \jb{\Pi_{+}u, R(\kk; u)\Pi_{+}u}_{L^2} 
\end{align*}
is finite, real-valued, and real-analytic as a function of $u$, and satisfies
\begin{align}
\frac{\dd \be}{\dd u} = -\big(m+\cj{m} +|m|^{2}\big).
 \label{bederiv}
\end{align}

\noi
Moreover, we have
\begin{align}
\{ \be(\kk;u); H_{\textup{BO}}(u)\}
= \{ \be(\kk;u); M(u)\} = 0, 
\label{becoms}
\end{align}

\noi
where $M(u)$ is the mass defined by 
\begin{align}
M(u) = \frac 12 \int_\M u^2 dx.
\label{be1}
\end{align}

\noi
Finally, by setting 
\begin{align}
 \be_{s}(\kk; u) 
 = \int_{\kk}^{\infty} \vk^{2s} \be(\vk;u)d\vk, 
\label{be2}
\end{align}

\noi
we have 
\begin{align}
C_{s}^{-1} \|u\|_{H^{s}_{\kk}}^{2} \leq \be_{s}(\kk; u) 
\leq C_{s} \|u\|_{H^{s}_{\kk}}^{2}. 
\label{besequiv}
\end{align}

\end{itemize}

\end{lemma}

Recalling that 
when $s < 0$, 
 the $H^s_\kk$-norm is decreasing in $\kk \ge 1$, 
 we see that if, 
given  $u\in H^{s}(\M)$, 
the condition  \eqref{kkcond} is satisfied
 for $\kk = \kk_0$ for some $\kk_0 \ge 1$, 
 then \eqref{kkcond} holds
 for any $\kk \ge \kk_0$.
In view of \eqref{becoms} and \eqref{be2}, 
we see that  $\be_s(\kk; u)$
 is conserved under the BO dynamics.

\subsection{Proof of Theorem \ref{THM:1}\,(i)}
In this subsection, we present the proof of Theorem \ref{THM:1}\,(i).
Let us state a lemma, where, under some assumption,  we control the growth of
$\be_s(\kk;u)$ via a Gronwall argument.

\begin{lemma}\label{LEM:bds1}
Let $\M = \R$ or $\T$ and $0<\dl<\infty$.
Given $- \frac 12 < s < 0$, let $C_s$ 
and $ \s=\tfrac{1}{2}( \tfrac{1}{2}+s)$
be as in Lemma~\ref{PROP:BOLAX}.
Let $u$ be a smooth global solution to 
ILW~\eqref{ILW}
such that 
\begin{align}
\kk \geq \sup_{ t\in [0,T]} C_{s} \big(1+c_0 \|u(t)\|_{H^{s}}\big)^{\frac{1}{2\s}}
\label{kkcondx}
\end{align}

\noi
for some  $\kk \ge 1$, $T>0$,  and $c_0\geq 1$.
Then, 
there exists $A_{\dl, s} > 0$, independent of 
 $\kk \ge 1$, $T>0$,  and $c_0\geq 1$, 
 such that 
\begin{align}
\be_{s}(\kk;u(t)) \leq e^{A_{\dl, s} t} \be_{s}(\kk; u(0)) \label{besT}
\end{align}

\noi
for any $0 \le t \le T$.

 \end{lemma}

In the following, we first present the proof of Theorem \ref{THM:1}\,(i)
by assuming Lemma \ref{LEM:bds1}
whose proof is presented at the end of this section.

\begin{proof}[Proof of Theorem \ref{THM:1}\,(i)]
We only consider the case $t \ge 0$.
Fix $T > 0$, and set 
\begin{align}
c_0=  C_se^{ A_{\dl, s}  T} \geq 1
\label{K1}
\end{align}

\noi
where $C_s> 0$ is as in Lemma \ref{PROP:BOLAX}
and $A_{\dl, s} > 0$ is as in Lemma \ref{LEM:bds1}.
Given $u_0 \in H^{\infty}(\M)$, 
fix $\kk \ge 1$ such that 
\begin{align}
\kk \geq C_{s}(1+2c_0 \|u_0\|_{H^{s}_{\kk}})^{\frac{1}{2\s}}, 
\label{kkcond1}
\end{align}

\noi
where 
$ \s=\tfrac{1}{2}( \tfrac{1}{2}+s)$
is  as in Lemma~\ref{PROP:BOLAX}.
Then, 
it follows from the continuity in time of $u$ 
with values in $H^{s}_{\kk}(\M)$ (recall $c_0 \ge 1$)
that  there exists $0<T_0 \leq T$ such that
\begin{align}
\|u(t)\|_{H^{s}_{\kk}}\leq 2c_0 \|u_0\|_{H^{s}_{\kk}}
\label{K2}
\end{align}

\noi
for any $0 \le t \le T_0$.
It follows from 
\eqref{kkcond1} and \eqref{K2}
that  the condition \eqref{kkcondx} 
in Lemma \ref{LEM:bds1}
is satisfied on $[0,T_0]$.
Hence, by applying 
Lemma \ref{LEM:bds1} with 
\eqref{besequiv} and \eqref{K1}, we have
\begin{align}
\| u(t)\|_{H^{s}_{\kk}}\leq c_0 \|u_0\|_{H^{s}_{\kk}}
 \label{Hskkbd}
\end{align}

\noi
for any $0 \le t \le T_0$.
Therefore, 
by  a continuity argument, we 
conclude that  \eqref{Hskkbd} holds on the entire interval $[0,T]$.

Finally, by choosing 
\begin{align}
\kk = C_{s} (1+2c_0 \|u_0\|_{H^{s}})^{\frac{1}{2\s}},
\label{K3}
\end{align}

\noi
we obtain from 
\eqref{HsHsk}, \eqref{K3}, 
\eqref{Hskkbd}, and 
the monotonicity of the $H^{s}_\kk$-norm  in $\kk$
that 
\begin{align*}
\begin{split}
\|u(t)\|_{H^{s}}
&  \leq \kk^{|s|}\|u(t)\|_{H^{s}_{\kk}} \\
& \le
C_s^{|s|+ 1}
e^{A_{\dl, s}T}
\Big(1+2C_s e^{A_{\dl, s}T} \|u_0\|_{H^{s}}\Big)^{\frac{2|s|}{1-2|s|}}
\|u_0\|_{H^{s}}
\end{split}
\end{align*}

\noi
for  $0\leq t\leq T$, from which 
we conclude 
\eqref{aprioribd} for any $t \ge  0$.
\end{proof}

We conclude this section by presenting the proof of Lemma \ref{LEM:bds1}.

\begin{proof}[Proof of Lemma \ref{LEM:bds1}]

Let $H_{\text{ILW}, \dl}$ be the Hamiltonian 
for ILW \eqref{ILW}
(with the Poisson bracket in~\eqref{P1}):
\begin{align*}
H_{\text{ILW}, \dl}(u) & = \frac{1}{2} \int_{\M} u \Gdl \dx u dx + \frac{1}{3}\int_{\mathcal{M}}u^3 dx. 
\end{align*}

\noi
In view of~\eqref{G1}, \eqref{Qdl}
 \eqref{Hamil}, and \eqref{be1},  
we then have
\begin{align}
H_{\textup{ILW},\dl}(u) = H_{\textup{BO}}(u) - \dl^{-1} M(u) + H_{\Qdl}(u), \quad \text{where } 
\ H_{\Qdl}(u) = \frac{1}{2}\int_{\M} u \Qdl u  dx.
 \label{HILW2}
\end{align}

\noi
Noting that $\Qdl$ is a symmetric operator, we have
\begin{align}
	\frac{\dd H_{\Qdl}}{\dd u} =  \Qdl u. \label{derivHQ}
\end{align}

Fix $\kk \ge 1$ and $T > 0$.
Let $u$ be a smooth global solution to \eqref{ILW}, 
 satisfying \eqref{kkcondx}.
Since $s<0$ and $\kk\geq 1$, we have $\|u(t)\|_{H^{s}_{\kk}}\leq \|u(t)\|_{H^{s}}$.
Thus, the hypothesis \eqref{kkcondx}
implies that  \eqref{kkcond} is satisfied.
In particular,
 $\be_{s}(\kk;u(t))$ is well defined for every $t\in [0,T]$ and all the results of Lemma~\ref{PROP:BOLAX} hold.
Recalling that 
$ \dt  F(u(t)) = \{ F , H_{\textup{ILW},\dl}\}(u(t))$
for a smooth function $F(u)$ (see \cite[Lemma~2.8]{Gre}), 
it follows from 
 \eqref{be2}, \eqref{HILW2}, \eqref{becoms}, 
and \eqref{P1} with 
\eqref{bederiv} and \eqref{derivHQ} that
\begin{align}
\begin{split}
\frac{d}{dt}&  \be_{s}(\kk;u(t)) 
 = \int_{\kk}^{\infty} \vk^{2s}  \frac{d}{dt} \be(\vk; u(t)) d\vk \\
& =  \int_{\kk}^{\infty} \vk^{2s} \{ \be(\vk) , H_{\textup{ILW},\dl}\}(u(t))d\vk \\
& =  \int_{\kk}^{\infty} \vk^{2s} \{ \be(\vk) , H_{\Qdl}\}(u(t))d\vk\\
& = -\int_{\kk}^{\infty} \vk^{2s} \int_{\M} (m(\tau;u(t))+\cj{m}(\tau;u(t))+|m(\tau;u(t))|^2) \Qdl \dx u(t)  dx d\vk \\
& =: I_1 +I_2 + I_3, 
\end{split}
\label{J1}
\end{align}

\noi
where $I_1$, $I_2$, and $I_3$
represent
the contributions from 
$m(\tau;u(t))$, $\cj{m}(\tau;u(t))$, and $|m(\tau;u(t))|^2$, respectively.
From 
Cauchy-Schwarz's inequality (on the Fourier side), \eqref{mbds}, 
Lemma \ref{LEM:Qdl}, \eqref{HsHsk}, and~\eqref{besequiv}, we have 
\begin{align}
\begin{split}
|I_1| + |I_2|& \les  \int_{\kk}^{\infty} \vk^{2s} \|m(\vk; u(t))\|_{H^{s}} \| \Qdl \dx u(t)\|_{H^{-s}} \,   d\vk\\
& \les  \dl^{-2}(1+\dl^{-2|s|}) \|u(t)\|_{H^s}^{2} \int_{\kk}^{\infty}  \vk^{2s-1}   d\vk \\
& \les \dl^{-2}(1+\dl^{-2|s|})\kk^{2|s|} \|u(t)\|_{H^{s}_{\kk}}^{2} \kk^{-2|s|}\\
& \les \dl^{-2}(1+\dl^{-2|s|}) \|u(t)\|_{H^{s}_{\kk}}^{2}\\
& \les \dl^{-2}(1+\dl^{-2|s|})  \be_{s}(\kk;u(t)).
\end{split}
\label{J2}
\end{align}

\noi
Similarly, 
from H\"older's inequality, 
the Sobolev embedding theorem (with small $\eps > 0$), 
Lemma~\ref{LEM:Qdl}, \eqref{HsHsk},
 $\jb{\xi}_\vk^{s+1} \ge \tau^{s+1}$ (recall that $s > - \frac 12$), 
\eqref{mbds}, 
and the monotonicity of the $H^s_\kk$-norm (in $\kk$), 
we have 
\begin{align}
\begin{split}
|I_3| 
& \le  \| \Qdl \dx u(t)\|_{H^{\frac{1}{2}+\eps}} \int_{\kk}^{\infty} \vk^{2s}\|m(\vk; u(t))\|_{L^2}^{2} d\vk \\
& \les \dl^{-2}(1+\dl^{-|s|- \frac 12 - \eps})\kk^{|s|} \|u(t)\|_{H^{s}_{\kk}} \int_{\kk}^{\infty} \vk^{2s-2(s+1)} \|m(\vk; u(t))\|_{H^{s+1}_{\vk}}^{2} \, d\vk \\
& \les \dl^{-2}(1+\dl^{-|s|- \frac 12 - \eps})
\kk^{|s|} \|u(t)\|_{H^{s}_{\kk}} \int_{\kk}^{\infty} \vk^{-2} \|u(t)\|_{H^{s}_{\vk}}^{2} \, d\vk\\
& \les  \dl^{-2}(1+\dl^{-|s|- \frac 12 - \eps})\kk^{|s|}  \|u(t)\|_{H^{s}_{\kk}}^{3} \int_{\kk}^{\infty} \vk^{-2} d\vk \\
& \les \dl^{-2}(1+\dl^{-|s|- \frac 12 - \eps})\frac{ \|u(t)\|_{H^{s}_{\kk}}^{3}}{ \kk^{1+s}}.
\end{split}
\label{J2a}
\end{align}

\noi
By separately considering the cases $\|u(t)\|_{H^{s}_{\kk}}<  1$
and $\|u(t)\|_{H^{s}_{\kk}}\ge  1$
(where, in the latter case,
we use  \eqref{kkcondx} with $2\s \leq 1+s$ which follows from  
 the definition of $\s$ 
in Lemma \ref{PROP:BOLAX}), the fact that $\kk, c_0 \ge 1$, and \eqref{besequiv}, we have 
\begin{align}
\begin{split}
|I_3| 
& \les \dl^{-2}(1+\dl^{-|s|- \frac 12 - \eps})  \|u(t)\|_{H^{s}_{\kk}}^{2} \\
& \les \dl^{-2}(1+\dl^{-|s|- \frac 12 - \eps}) \be_{s}(\kk;u(t)).
\end{split}
\label{J3}
\end{align}

\noi
Hence, from \eqref{J1}, \eqref{J2}, and \eqref{J3}, 
we have 
\begin{align*}
\frac{d}{dt} \be_{s}(\kk;u(t))
 \le C \dl^{-2}(1+\dl^{-|s|- \frac 12 - \eps})
 \be_{s}(\kk;u(t))=:
 A_{\dl, s}\be_{s}(\kk;u(t)), 
\end{align*}

\noi
where $A_{\dl, s}$ is independent of 
 $\kk \ge 1$, $T>0$,  and $c_0\geq 1$.
Then, the desired bound \eqref{besT}
follows from Gronwall's inequality.
\end{proof}

\section{Ill-posedness}

In this section, we 
prove ill-posedness of ILW \eqref{ILW}
in $H^s(\M)$ for $s < -\frac 12$
(Theorem~\ref{THM:1}\,(ii)).
In Subsection \ref{SEC:Rcase}, we discuss the real line case, 
while we treat the periodic case in 
Subsection~\ref{SEC:Tcase}.

\subsection{Ill-posedness on the real line}
\label{SEC:Rcase}

We first recall the following traveling wave
solutions for ILW \eqref{ILW}; see \cite{joseph, AT, Albert}.
Given $c > 0$, 
let 
$a=a(c)\in \big(0, \frac\pi\dl\big)$ be the unique solution of the equation
\begin{align*}
	a\dl \cot(a\dl) =1-c\dl.
\end{align*}

\noi
Then, $u_c$ defined by 
\begin{align}
u_{c}(t,x) = \frac{-a \sin (a\dl)}{\cosh(a(x-ct))+\cos(a\dl)}
\label{ILWtravel}
\end{align}

\noi
satisfies \eqref{ILW}.
Our strategy is to follow  the work \cite{BL} by 
Biagioni and Linares for the BO equation
and to take $c \to \infty$
(which is  equivalent to $\frac{\pi}{a}\to \dl$).
From  the formula \cite[(6) on p.\,30]{Bateman}:
\begin{align*}
\int_{\R} \frac{e^{-i\xi x}}{ \cosh(ax) + \cos(a\dl)}dx 
= \frac{2\pi}{a \sin(a\dl)}\frac{\sinh(\dl \xi)}{\sinh(\frac{\pi \xi}{a})}, 
\end{align*}

\noi
we see that 
\begin{align}
	\ft u_c (0,\xi) = -\sqrt{2\pi} \frac{\sinh(\dl \xi)}{\sinh(\frac{\pi \xi}{a})}
	 \label{FTuc}
\end{align} 

\noi
with the understanding that
\begin{align}
\frac{\sinh(\dl \xi)}{\sinh(\frac{\pi \xi}{a})}\bigg|_{\xi = 0}
=  \lim_{\xi \to 0}\frac{\sinh(\dl \xi)}{\sinh(\frac{\pi \xi}{a})}
= \frac{a\dl}{\pi}.
\label{frac1}
\end{align}

\noi
Observe that $\ft u_c (0,\xi)$ enjoys the following properties:~(i)~it is bounded for $ |\xi|\leq 1$, 
(ii)~since $\dl<\frac{\pi}{a}$, it decays exponentially as $|\xi| \to \infty$, and (iii)
for each fixed $\xi\in \R$, we have 
\begin{align*}
	\ft u_c (0,\xi)  \to -\sqrt{2\pi}
\end{align*} 

\noi
 as $c\to \infty$
 (i.e.~$\frac{\pi}{a}\to \dl$).
 These three properties
with the dominated convergence theorem
  ensure that,\footnote{Recall our convention \eqref{FT1}
  for the Fourier transform.}
  as $c \to \infty$, 
$u_c(0)  \to -2\pi \dl_0$ in $H^{s}(\R)$ for any $s<-\frac{1}{2}$, 
where $\dl_0$ denotes the Dirac delta function on $\R$.
Together with 
\eqref{ILWtravel}, this convergence implies
\begin{align}
\lim_{c\to \infty} \|u_{c}(t)\|_{H^{s}(\R)}=
\lim_{c\to \infty} \|u_{c}(0)\|_{H^{s}(\R)}
= 
 2\pi \|\dl_0\|_{H^{s}(\R)} \label{ILWill1}
\end{align}

\noi
for any  $t\in \R$. On the other hand, for any test function $\psi\in C_{c}^{\infty}(\R)$
and any fixed $t\neq 0$, 
we have 
\begin{align*}
\jb{u_{c}(t),\psi}_{L^2} & = \int_{\R} u_{c}(0,x) \psi(x+ct)dx =  \int_{\R}  \frac{- \sin(a\dl)}{\cosh(x)+\cos(a\dl)} \psi(\tfrac{x}{a}+ct)dx\to 0
\end{align*}
as $c\to \infty$,
since $u_c(0)$ decays exponentially
as $|x| \to \infty$.
In particular, 
for $t\ne 0$, 
 $u_c(t)$  converges to 0 in the distributional sense
as $c \to \infty$, 
which 
implies that $u_c(t)$ does not converge in $H^s(\R)$ in view of \eqref{ILWill1}.
This completes the proof of Theorem \ref{THM:1}\.(ii) in the real line case.

\subsection{Ill-posedness on the circle}\label{SEC:Tcase}

We go over  the following construction of a periodic traveling wave solution for \eqref{ILW} in \cite{Miloh}. 
We start with 
the profile
 $u_c(0)$ in \eqref{ILWtravel} 
 for the traveling wave solution on the real line
and apply the Poisson summation formula 
 which, with our convention for the Fourier transforms, reads as 
\begin{align}
\sum_{n\in \Z} f(x+n) = \sqrt{2\pi}\sum_{\xi \in 2\pi\Z}  \F_{\R} (f)(\xi)e^{i\xi x}, 
\label{Poisson}
\end{align}
where $f:\R \to \R$ and its Fourier transform  $\F_\R( f)$ decay sufficiently rapidly. 
Let $U_c$ be the periodization of $u_c(0)$ in \eqref{ILWtravel}.
Then, from \eqref{Poisson}
and  \eqref{FTuc},  we have 
\begin{align}
U_{c}(x) 
= \sum_{n\in \Z} \frac{- a \sin (a\dl)}{ \cosh(a(x+n))+\cos(a\dl)}
= -2\pi  \sum_{\xi \in 2\pi \Z} \frac{\sinh(\dl \xi)}{\sinh(\frac{\pi \xi}{a})} e^{i\xi x} , \label{phic}
\end{align}
where $a\in \big(0,\frac \pi \dl\big)$.
Here, we used the convention
 \eqref{frac1}.

 We now verify that there exists a choice 
 for both  $c = c(\dl, a)$ and a constant of integration 
 $B = B(\dl, a)$ such that $U_c$ solves 
\begin{align}
-cU_c+\dl^{-1}U_c -\Tdl \dx U_c -U_c^{2} = B,
 \label{traveleqn}
\end{align}

\noi
\noi
where $\Tdl$ is as in \eqref{Tdl}.
Fix $x\in \R$. For $n\in \Z$, define $b_n = b_n(x)$ and $d_n = d_n(x)$ by 
\begin{align}
b_{n}  = \frac{1}{\cosh(a(x+n))+\cos(a\dl)} \quad \text{and} \quad d_n  = b_{n} \sinh(a(x+n)).
\label{B1}
\end{align}

\noi
Note that 
\begin{align}
U_c = - a\sin(a\dl)\sum_{n\in \Z} b_{n}. 
\label{phic2}
\end{align}

\noi
From \eqref{phic} and \eqref{Tdl}, we have
\begin{align}
\Tdl U_c(x) = 2\pi i  \sum_{\xi \in 2\pi \Z} \frac{\cosh(\dl \xi)}{\sinh(\frac{\pi}{a}\xi)} e^{i\xi x} = -\sum_{n\in \Z} \frac{a \sinh(a(x+n))}{\cosh(a(x+n))+\cos(a\dl)}, 
\label{B2}
\end{align}

\noi
where the second equality follows from 
\eqref{Poisson}
and \cite[(7) on p.\,88]{Bateman}.
Thus, from \eqref{B2} with~\eqref{B1}, we have 
\begin{align}
\begin{split}
-\Tdl \dx U_c(x)
& =  a^2 \sum_{n\in \Z} \frac{1+\cos(a\dl) \cosh(a(x+n))}{\big[\cosh(a(x+n)) +\cos(a\dl)\big]^{2}}\\
&  =  \sum_{n\in \Z} a^2\big[1+\cos(a\dl) \cosh(a(x+n))\big] b_{n}^{2}.
\label{B2b}
\end{split}
\end{align}

Next,  we compute $U_c^{2}$. 
By \eqref{phic2}, we have 
\begin{align}
U_c^{2} = \sum_{n\in \Z}  a^{2} \sin^{2}(a\dl) b_{n}^{2} +  a^{2}\sin^{2}(a\dl) \sum_{\substack{n,m\in \Z \\ n\neq m}} b_{n}b_{m}.
\label{B2a}
\end{align}

\noi
In order to compute the second term above, we use the following identity 
(see \cite[(A\,1) on p.\,622]{Miloh}): 
\begin{align}
\begin{split}
2b_{n} b_{m} &  = -\frac{\cos(a\dl)}{\sinh^{2}(\frac{a}{2}(m-n))+\sin^2 (a\dl)}(b_n+b_m)\\
& \quad  +\frac{ \coth(\frac{a}{2}(n-m))}{ \sinh^{2}(\frac{a}{2}(m-n))+\sin^2 (a\dl)}(d_{n}-d_{m})
\end{split}
\label{bnbm}
\end{align}

\noi
for all $n,m\in \Z$, $n\neq m$.
Putting $n=k+\l$ and $m=k-\l$, it follows from \eqref{bnbm} that 
\begin{align}
\begin{split}
2\sum_{\substack{n,m\in \Z \\ n\neq m}} b_{n}b_{m}  
& = -\lim_{N\to \infty} \sum_{\l \in \Z\setminus\{0\}} \frac{ \cos(a\dl)}{\sinh^{2}(\frac{a}{2}\l)+\sin^2(a\dl)} \sum_{k=-N}^{N} ( b_{k+\l}+b_{k-\l}) \\
& \quad + \lim_{N\to \infty} \sum_{\l \in \Z\setminus\{0\}} \frac{\coth(\frac{a}{2}\l)}{\sinh^{2}(\frac{a}{2}\l)+\sin^{2}(a\dl)} \sum_{k=-N}^{N} (d_{k+\l}-d_{k-\l}).
\end{split}
 \label{bnbm2}
\end{align}

\noi
Since $b_{k}>0$ for any $k\in \Z$, the monotone convergence theorem implies 
\begin{align}
\begin{split}
& \lim_{N\to \infty} \sum_{\l \neq 0} \frac{ \cos(a\dl)}{\sinh^{2}(\frac{a}{2}\l)+\sin^2(a\dl)} \sum_{k=-N}^{N} ( b_{k+\l}+b_{k-\l}) \\
& \quad = 4 \bigg( \sum_{\l=1}^{\infty} \frac{ \cos(a\dl)}{\sinh^{2}(\frac{a}{2}\l)+\sin^2(a\dl)} \bigg) \sum_{k\in \Z} b_{k}.
\end{split}
\label{B3}
\end{align}

\noi
As for the second term on 
the right-hand side of \eqref{bnbm2}, 
by noting that, for each $x \in \R$,    $\text{sgn}(k)d_k(x) \to 1$ as $|k|\to \infty$, 
we have, for each fixed $\l \in \Z\setminus\{0\}$ and $x \in \R$, 
\begin{align}
\sum_{k=-N}^{N} (d_{k+\l}(x)-d_{k-\l}(x)) = \sum_{k=N-\l+1}^{N+\l} d_{k}(x) - \sum_{k=-N-\l}^{-N+\l-1}d_{k}(x) \too 4 \l, 
\label{B4}
\end{align}

\noi
as $N\to \infty$. 
Putting 
\eqref{bnbm2}, \eqref{B3}, and \eqref{B4} together, we have 
\begin{align}
\begin{split}
\sum_{\substack{n,m\in \Z \\ n\neq m}} b_{n} b_{m} &= -2 \bigg( \sum_{\l=1}^{\infty} \frac{ \cos(a\dl)}{\sinh^{2}(\frac{a}{2}\l)+\sin^2(a\dl)} \bigg) \sum_{k\in \Z} b_{k}  + 4\sum_{\l=1}^{\infty} \frac{\l\coth(\frac{a}{2}\l)}{\sinh^{2}(\frac{a}{2}\l)+\sin^{2}(a\dl)} , 
\end{split}
\label{B5}
\end{align}

\noi
where we suppressed the $x$-dependence.
See \cite[(18) on p.\,622]{Miloh}. 
Hence, from \eqref{B2a}, \eqref{B5}, and~\eqref{phic2}, we obtain
\begin{align}
U_c^{2} =  \sum_{n\in \Z}  a^{2} \sin^{2}(a\dl) b_{n}^{2} +V U_c + D, 
\label{B6}
\end{align}

\noi
where $V = V(\dl, a)$
and $D = D(\dl, a)$ are given by 
\begin{align}
\begin{split}
 V& =\sum_{\l=1}^{\infty} \frac{ a\sin(2a\dl)}{\sinh^{2}(\frac{a}{2}\l)+\sin^2(a\dl)}, \\
 D & = 4 a^2 \sin^{2}(a\dl)\sum_{\l=1}^{\infty} \frac{\l\coth(\frac{a}{2}\l)}{\sinh^{2}(\frac{a}{2}\l)+\sin^{2}(a\dl)}.
 \end{split}
\label{B6a} 
\end{align}

By substituting 
\eqref{B2b}
and  
\eqref{B6}
into 
 \eqref{traveleqn}, we obtain
\begin{align*}
(-c+\dl^{-1}-V)U_c + \sum_{n\in \Z} a^2\big[ 1+\cos(a\dl)\cosh(a(x+n))-\sin^{2}(a\dl)  \big]b_n^2 =B+D.
\end{align*}

\noi
Using \eqref{phic2} with \eqref{B1}, this becomes
\begin{align}
\begin{split}
& \sum_{n\in\Z} ab_{n}^{2} \big[ a+ a \cos(a\dl)\cosh(a(x+n))
- a \sin^2(a\dl)\\
&\quad -(-c+\dl^{-1}-V)\sin(a\dl)( \cosh(a(x+n))+\cos(a\dl))\big]  = B+D.
\end{split}
\label{B7}
\end{align}

\noi
Noting that the right-hand side of \eqref{B7}
is independent of $x \in \T$ (but still depends on $\dl$ and $a$), 
we now impose the following three conditions:
\begin{align}
\begin{split}
a \cos(a\dl) &= (-c+\dl^{-1}-V)\sin(a\dl),  \\
a- a \sin^{2}(a\dl)& =(-c+\dl^{-1}-V)\sin(a\dl)\cos(a\dl), \\
B& =-D.
\end{split}
\label{cond1}
\end{align}

\noi
Note that the last condition in \eqref{cond1}
should be interpreted as a definition of $B = B(\dl, a)$
in terms of $D = D(\dl, a)$ in \eqref{B6a}.
In view of the first condition in \eqref{cond1}, we choose $c$ such that 
\begin{align}
-c+\dl^{-1}-V=a\cot(a\dl).
\label{B8}
\end{align}

\noi
such that the first condition in \eqref{cond1} is satisfied.
It is easy to check that with this choice of $c$, 
the second condition in \eqref{cond1} is also satisfied.
From \eqref{B6a} and \eqref{B8}, we have
\begin{align*}
c= \dl^{-1} -a\cot(a\dl) -\sum_{\l=1}^{\infty} \frac{a\sin(2a\dl)}{\sinh^{2}(\frac{a}{2}\l)+\sin^{2}(a\dl)}, 
\end{align*}

\noi
which shows that $c\to \infty$, 
as $a\to \frac{\pi}{\dl}$.

Having constructed the periodic traveling wave 
$u_c (t, x) := U_c(x - ct)$ (with $c = c(\dl, a)$) as above, 
we can proceed as  in Subsection \ref{SEC:Rcase}
to prove ill-posedness of ILW \eqref{ILW} on the circle. 
In view of~\eqref{phic}, 
by arguing as in Subsection \ref{SEC:Rcase}, 
we see that, as $c\to \infty$
(i.e.~$a\to \frac{\pi}{\dl}$), 
 $u_c|_{t = 0} = U_c$ converges to
 $-2\pi\dl_0$
  in $H^{s}(\T)$
  for $s<-\frac 12$, 
  where $\dl_0$ is 
  the Dirac delta function on $\T$. 
 On the other hand, we have  
  \begin{align*}
  \ft u_c(t, 2\pi)
  =  
  \int_{\T}  U_c(x-ct)e^{-2\pi i x}dx
= -2\pi e^{-2\pi i ct} \frac{\sinh(2\pi \dl)}{\sinh( 2\pi\frac{\pi}{a})}.
\end{align*}
As $a\to \frac{\pi}{\dl}$ (and thus $c \to \infty$), 
we have 
$ \frac{\sinh(2\pi \dl)}{\sinh( 2\pi\frac{\pi}{a})}\to 1$ but the  exponential $e^{-2\pi i ct} $ diverges
for $t \ne 0$.
Hence, for $t \ne 0$, 
$u_c(t)$ does not converge in the distributional sense
(and in particular in $H^s(\T)$).
This proves ill-posedness of ILW \eqref{ILW}
in $H^s(\T)$ for $s < -\frac 12$.

Next, let us briefly discuss ill-posedness in 
$H^s_\al(\T)$ for given $\al \in \R$.
Given $c > 0$, 
let $\mu_c$ denote the spatial mean of $U_c$ in \eqref{phic}.
In view of  \eqref{frac1}, we have
\begin{align}
\mu_c = -2a \dl \too -2\pi = \text{the spatial mean of $-2\pi \dl_0$}, 
\label{ave1}
\end{align}

\noi
as $c\to \infty$ (and hence $a\dl \to \pi$).
Given $\g \in \R$, define a Galilean transform $\G_\g$ by 
\begin{align}
 \G_{\g}(u) (t, x) = u (t, x- 2\g t) - \g.
\label{ave2}
\end{align}

\noi
Note that if $u$ is a solution to \eqref{ILW}, 
then so is $\G_\g(u)$ for any $\g \in \R$.

Fix $\al \in \R$.
Given $c > 0$, let $u_c (t, x) = U_c(x - ct)$
be the traveling wave solution constructed above.
Then, 
by setting 
$v_{c, \al} = \G_{\mu_c - \al}(u_c)$, 
it follows from the discussion above with \eqref{ave1}
and~\eqref{ave2}
that (i)~$v_{c, \al}(t) \in H^s_\al(\T)$ for any $t \in \R$, 
(ii)~$v_{c, \al}(0)$ converges to 
 $-2\pi\dl_0 + (2\pi + \al)$
  in $H^{s}_\al(\T)$
  for $s<-\frac 12$, 
and (iii) we have 
  \begin{align*}
\ft v_{c, \al}(t, 2\pi) 
 = -2\pi e^{-2\pi i (c+2\mu_c - \al)t} \frac{\sinh(2\pi \dl)}{\sinh( 2\pi\frac{\pi}{a})}
\end{align*}

\noi
which is divergent as $c\to \infty$
for any $t \ne 0$. 
This proves ill-posedness of 
 ILW \eqref{ILW}
in $H^s_\al(\T)$ for $s < -\frac 12$ and $\al \in \R$.
This concludes the proof of 
 Theorem \ref{THM:1}\,(ii)
on the circle.

\begin{ackno}\rm 
	This material is based upon work supported by the Swedish
	Research Council under grant no.~2016-06596
	while 
the first, fourth, and fifth authors were in residence at 
Institut  Mittag-Leffler in Djursholm, Sweden
during the 
program ``Order and Randomness in Partial Differential Equations''
in Fall, 2023.
A.C., J.F., G.L., and T.O.~were supported by the European Research Council
(grant no. 864138 ``SingStochDispDyn'').
G.L. was also supported by the EPSRC New Investigator Award (grant no.~EP/S033157/1).
D.P. was supported by a Trond Mohn Foundation grant.
\end{ackno}


\begin{thebibliography}{99}	
	
		\bibitem{ABFS}
L.~Abdelouhab, J.L.~Bona, M.~Felland, J.-C.~Saut, 
{\it Nonlocal models for nonlinear, dispersive waves}, 
Phys. D 40 (1989), no. 3, 360--392.



\bibitem{Albert}
J.P.~Albert, 
{\it Positivity properties and uniqueness of solitary wave solutions of the intermediate long-wave equation},
 Evolution equations (Baton Rouge, LA, 1992), 11--20, Lecture Notes in Pure and Appl. Math., 168, Dekker, New York, 1995.

\bibitem{AT}
J.P.~Albert, J.F.~Toland, 
{\it On the exact solutions of the intermediate long-wave equation},
 Differential Integral Equations 7 (1994), no. 3-4, 601--612. 
	
	\bibitem{AH}
	J.~Angulo Pava, S.~Hakkaev, {\it Ill-posedness for periodic nonlinear dispersive
equations}, 
 Electron. J. Differential Equations 2010, No. 119, 19 pp. 

	
	


\bibitem{BL}
H.~Biagioni, F.~Linares, 
{\it Ill-posedness for the derivative Schr\"odinger and generalized Benjamin-Ono equations},
 Trans. Amer. Math. Soc. 353 (2001), no. 9, 3649--3659.





	

	
	\bibitem{CLOP}
	A.~Chapouto, G.~Li, T.~Oh, D.~Pilod,
	{\it Deep-water limit of 
		the intermediate long wave
		equation in $L^2$}, 
arXiv:2311.07997 [math.AP].

\bibitem{CLOZ}
A. Chapouto, G. Li, T. Oh, 
\textit{Deep-water and shallow-water limits of statistical equilibria for the
intermediate long wave equation}, in preparation.

	
	
	\bibitem{Bateman}
A.~Erd\'elyi, W.~Magnus, F.~Oberhettinger, F. G.~Tricomi, 
{\it Tables of integral transforms. Vol. I.
Based, in part, on notes left by Harry Bateman},
McGraw-Hill Book Co., Inc., New York-Toronto-London, 1954. xx+391 pp.

	
	

%
	
	
	\bibitem{GKT}
	P. G\'erard, T. Kappeler, P. Topalov, 
	\textit{Sharp well-posedness results of the Benjamin-Ono equation in $H^s(\T;\R)$ and qualitative properties of its solution}, 
 Acta Math. 231 (2023), no. 1, 31--88.
	

\bibitem{Gre}
B.~Gr\'ebert, {\it Birkhoff normal form and Hamiltonian PDEs},
 Partial differential equations and applications, 1--46, S\'emin. Congr., 15, Soc. Math. France, Paris, 2007.

	
	
	\bibitem{IS23}
	M.~Ifrim, J.-C.~Saut,
	{\it The lifespan of small data solutions for Intermediate Long Wave equation (ILW)}, arXiv:2305.05102 [math.AP].
	
\bibitem{IT1}
M.~Ifrim, D.~Tataru, 
{\it Well-posedness and dispersive decay of small data solutions for the Benjamin-Ono equation},
 Ann. Sci. \'Ec. Norm. Sup\'er.  52 (2019), no. 2, 297--335.


	

\bibitem{IK}
A.~Ionescu, C.~Kenig, 
{\it Global well-posedness of the Benjamin-Ono equation in low-regularity spaces}, 
J. Amer. Math. Soc. 20 (2007), no. 3, 753--798. 

	
\bibitem{joseph}
R.I.~Joseph, {\it Solitary waves in a finite depth fluid,} J. Phys. A 10 (1977), no. 12, 225--227.
	


	\bibitem{KPV3}
	C.~Kenig, G.~Ponce, L.~Vega, {\it On the ill-posedness of some canonical dispersive equations}, Duke Math.
	J. 106 (2001), no. 3, 617--633.
%


	\bibitem{KLV}
	R.~Killip, T.~Laurens, M.~Vi\c{s}an, 
	{\it Sharp well-posedness of for the Benjamin-Ono equation},
		arXiv:2304.00124 [math.AP].
	
	
	
	
	
	\bibitem{KV}
	R.~Killip, M.~Vi\c{s}an,
	{\it KdV is well-posed in $H^{-1}$},
 Ann. of Math.  190 (2019), no. 1, 249--305. 
	
	
	\bibitem{KVZ}
	R.~Killip, M.~Vi\c{s}an, X.~Zhang, 
	{\it Low regularity conservation laws for integrable PDE},
	Geom. Funct. Anal. 28 (2018), no. 4, 1062--1090. 
	
	

%

\bibitem{KlausSchippa}
F.~Klaus, R.~Schippa, {\it A priori estimates for the derivative nonlinear Schrödinger equation}, Funkcial. Ekvac. 65 (2022), no. 3, 329--346.

	\bibitem{KS}
	C.~Klein, J.-C. Saut, {\it Nonlinear dispersive equations–inverse scattering and PDE methods}, Applied Mathematical Sciences, 209. Springer, Cham, [2021], \textcopyright2021. xx+580 pp.

\bibitem{KT}
H.~Koch, D.~Tataru, 
{\it Energy and local energy bounds for the 1-d cubic NLS equation in $H^{-\frac 14}$}, 
Ann. Inst. H. Poincar\'e C Anal. Non Lin\'eaire 29 (2012), no. 6, 955--988. 

\bibitem{KTz2}
H.~Koch, N.~Tzvetkov, 
{\it Nonlinear wave interactions for the Benjamin-Ono equation},
Int. Math. Res. Not. 2005, no. 30, 1833--1847. 


	\bibitem{KKD}
	T.~Kubota, D.R.S.~Ko, L.D.~Dobbs, {\it Weakly-nonlinear, long internal gravity waves in stratified fluids of
	finite depth}, J. Hydronautics 12 (1978), no. 4, 157--165.
	
	
	
	
	
	
	\bibitem{TL1}
	T.~Laurens,
	{\it Global well-posedness for $H^{-1}(\R)$ perturbations of KdV with exotic spatial asymptotics,}
	Comm. Math. Phys. 397
(2023), no. 3, 1387--1439.

\bibitem{Gli}
G.~Li, {\it Deep-water and shallow-water limits of the generalised intermediate long wave equation}, 
arXiv:2207.12088 [math.AP].

	
	\bibitem{LOZ}
	G.~Li, T.~Oh, G.~Zheng, 
	{\it On the deep-water and shallow-water limits of the intermediate long wave equation from a statistical viewpoint}, arXiv:2211.03243 [math.AP].
	
	

		
\bibitem{Miloh}		
	T.~Miloh, {\it On periodic and solitary wavelike solutions of the intermediate long-wave equation}, 
J. Fluid Mech. 211 (1990), 617--627. 
	

\bibitem{Moli2}
L.~Molinet, 
{\it Global well-posedness in $L^2$ for the periodic Benjamin-Ono equation}, 
Amer. J. Math. 130 (2008), no. 3, 635--683. 

	
	\bibitem{MP}
L.~Molinet, D.~Pilod, 
{\it The Cauchy problem for the Benjamin-Ono equation in $L^2$ revisited},
 Anal. PDE 5 (2012), no. 2, 365--395. 
	

\bibitem{MPV}
L.~Molinet, D.~Pilod, 
S.~Vento, {\it On well-posedness for some dispersive perturbations of Burgers' equation}, 
Ann. Inst. H. Poincar\'e C Anal. Non Lin\'eaire 35 (2018), no. 7, 1719--1756.




\bibitem{MST}
L.~Molinet, 
J.-C.~Saut, N.~Tzvetkov, 
{\it Ill-posedness issues for the Benjamin-Ono and related equations},
SIAM J. Math. Anal. 33 (2001), no. 4, 982--988. 




\bibitem{MV}
L.~Molinet, S.~Vento, 
{\it Improvement of the energy method for strongly nonresonant dispersive equations and applications},
 Anal. PDE 8 (2015), no. 6, 1455--1495.




%


\bibitem{OW2}	
T.~Oh, Y.~Wang,
{\it Global well-posedness of the periodic cubic fourth order NLS in negative Sobolev spaces},
 Forum Math. Sigma 6 (2018), e5, 80 pp. 	


	
\bibitem{OW1}	
T.~Oh, Y.~Wang,
{\it Global well-posedness of the one-dimensional cubic nonlinear Schrödinger equation in almost critical spaces}, J. Differential Equations 269 (2020), no. 1, 612--640.
	
	
	

	
	\bibitem{S19}
	J.-C. Saut, {\it Benjamin-Ono and intermediate long wave equations: modeling, IST and PDE. Nonlinear
	dispersive partial differential equations and inverse scattering}, 95--160, Fields Inst. Commun., 83, Springer,
	New York, [2019], \textcopyright2019.
	
\bibitem{Talbut}
 B.~Talbut, {\it Low regularity conservation laws for the Benjamin-Ono equation}, 
  Math. Res. Lett. 28 (2021), no. 3, 889--905. 
	
	

	
	
	
\end{thebibliography}
\end{document}